\documentclass[12pt]{article}

\usepackage{amsmath,amsthm,amssymb,xcolor,graphicx,subcaption,comment,authblk}
\usepackage[TS1,T1]{fontenc}
\usepackage[pagewise]{lineno}
\usepackage{multirow}
\usepackage[normalem]{ulem}

\newtheorem{thm}{Theorem}[section]
\newtheorem{lem}{Lemma}[section]

\newtheorem{Def}{Definition}[section]

\newcommand{\R}{\mathbb{R}}
\newcommand{\N}{\mathbb{N}}
\newcommand{\p}{\partial}

\newcommand{\M}{\mathcal{M}}

\newcommand{\xM}{x_{\max}}

\newcommand{\support}{\text{supp}}
\newcommand{\rint}{\int_{\R^+}}
\newcommand{\sumj}{\sum_{j=1}^{J}}
\newcommand{\sumi}{\sum_{i=1}^{J}}

\begin{document}

\title{\textbf{High Resolution Finite Difference Schemes for a Size Structured Coagulation-Fragmentation Model in the
Space of Radon Measures}}

\author[$\dagger$]{Azmy S. Ackleh} 
\author[$\dagger,\star$]{Rainey Lyons}
\author[$\ddagger$]{Nicolas Saintier}

\affil[$\dagger$]{Department of Mathematics, University of Louisiana at Lafayette,\ Lafayette, Louisiana 70504, U.S.A.\\ackleh@louisiana.edu}

\affil[$\star$]{Department of Mathematics and Computer Science, Karlstad University, 651 88 Karlstad, Sweden \\ rainey.lyons@kau.se}

\affil[$\ddagger$]{Departamento de Matem\'atica, Facultad de Ciencias Exactas y Naturales,  
Universidad de Buenos Aires\\ (1428) Pabell\'on I - Ciudad Universitaria - Buenos Aires - ARGENTINA\\nsaintie@dm.uba.ar}

\maketitle

\abstract{In this paper we develop explicit and semi-implicit second-order high-resolution  finite difference schemes for a structured coagulation-fragmentation model formulated on the space of Radon measures. We prove the convergence of each of the two schemes to the unique weak solution of the model. We perform numerical simulations to demonstrate that the second order accuracy is achieved by both schemes.  }



\section{Introduction}\label{SecIntro}

Coagulation-fragmentation (CF) equations have been used to model many physical and biological phenomenon \cite{BurdJack,Ald}. In particular, when combined with transport terms, these equations can be used to model the population dynamics of oceanic phytoplankton \cite{AckFitz,Ack,RudWie}. Setting such models in the space of Radon measures allows for the unified study of both discrete and continuous structures. Not only are the classical discrete and continuous CF equations special cases of the measure valued model (as shown in \cite{AckLyonSaint2}), but this setting allows for a mixing of the two structures which has become of interest in particular applications \cite{BairdSuli1,BairdSuli2}. 

With the above applications in mind, numerical schemes to solve CF equations are of great importance to researchers. In particular, finite difference methods offer numerical schemes which are easy to implement and approximate the solution with a high order of accuracy. The latter benefit is especially important in the study of stability and optimal control of such equations. 

The purpose of this article is to make improvements on the two of the schemes presented in \cite{AckLyonSaint3}, namely the fully explicit and semi-implicit schemes. These schemes are shown to have certain advantages and disadvantages discussed in the aforementioned study. In particular, the fully explicit scheme has the qualitative property of conservation of mass through coagulation. On the other hand, the semi-implicit scheme has a more relaxed Courant–Friedrichs–Lewy (CFL) condition which does not depend on the initial condition. We have decided not to attempt to improve the third scheme presented in \cite{AckLyonSaint3} as there does not seem to be a significant advantage of the named conservation law scheme to outweigh the drastic computational cost. 

As the state space is highly irregular, the improvement of these schemes must be handled with care. As shown in \cite{LeV}, discontinuities and singularities in the solution can cause drastic changes in not only the order of convergence of the scheme, but also in the behavior of the scheme. Such phenomenon is demonstrated in \cite{AckLyonSaint, LeV}. To handle these issues, we turn to a high resolution scheme studied with classical structured population models (i.e. without coagulation-fragmentation) in \cite{Shen, AckVin, AckLyonSaint}. This scheme makes use of a minmod flux limiter to control any oscillatory behavior of the scheme caused by irregularities.

The layout of the paper is as follows. In Section \ref{SecPrelim} we present any notation and preliminary results about the model and state space used throughout the paper. In Section \ref{SecModel} we describe the model and state all assumptions imposed on the model parameters. In Section \ref{SecNumMeth}, we present the numerical schemes, their CFL conditions and state the main Theorem of the paper. Finally, we test the convergence rate of the schemes against well-known examples in Section \ref{Sec_Numerics}.

\section{Notation}\label{SecPrelim}

We make use of standard notations for function spaces. The most common examples of these are $C^1(\R^+)$ for the space of real valued continuously differentable functions and $W^{1,\infty}(\R^+)$ for the usual Sobelov space. The space of Radon measure will be denoted with $\M(\R^+)$ with $\M^+(\R^+)$ representing it positive cone. This space will be equipped with the Bounded-Lipschitz (BL) norm given by 
\[ \|\mu \|_{BL} := \sup_{\|\phi\|_{W^{1,\infty}}\leq 1} \left \lbrace \rint \phi(x) \mu(dx) : \phi \in W^{1,\infty}(\R^+) \right \rbrace .\]
Another norm of interest to this space is the well studied Total Variation (TV) norm given by 
\[ \Vert \nu \Vert_{TV} = |\nu|(\R^+) = \sup_{\| f \|_{\infty}\leq 1} \left\lbrace\int_{\R^+}f(x) \nu(dx) : f\in C_c(\R^+)  \right\rbrace. \]
For more information about these particular norms and their relationship we direct the reader to \cite{GwCzTh, DulGwBook}. 
For lucidity, we use operator notation in place of integration when we believe it necessary, namely 
\[(\mu,f) := \int_{A} f(x) \mu(dx),\] 
where the set $A$ is the support of the measure $\mu$.
Finally, we denote the minmod function by $\text{mm}(a,b)$ and use the following definition
\[ \text{mm}(a,b) := \frac{\text{sign}(a) + \text{sign}(b)}{2}\max(|a|,|b|).\]

\section{Model and Assumptions}\label{SecModel}

The model of interest is the size-structured coagulation fragmentation model given by
\begin{equation}\label{CoagModel}
\left \lbrace
\begin{split}
&\partial_t \mu + \partial_x(g(t,\mu)\mu) + d(t,\mu)\mu = K[\mu] + F[\mu], \qquad (t,x) \in (0,T)\times (0,\infty), \\
&g(t,\mu)(0) D_{dx}\mu (0) = \int_{\R^+}\beta(t,\mu)(y) \mu(dy), \hspace{1.7cm}  t\in [0,T], \\ 
&\mu(0)=\mu_0 \in \M^+(\R^+),
\end{split}. \right.
\end{equation} 
where $\mu(t) \in \M^+(\R^+)$ represents individuals' size distribution at time $t$ and the functions $g, d, \beta$ are their growth, death, and reproduction rate. 
The coagulation and fragmentation processes of a population distributed according to $\mu \in \M^+(\R^+)$ are modeled by the measures $K[\mu]$ and $F[\mu]$ 
given 
$$ (K[\mu],\phi) = \frac12 \rint \rint \kappa(y,x)\phi(x+y)\,\mu(dx)  \,\mu(dy) - \rint \rint \kappa(y,x)\phi(x)\,\mu(dy) \,\mu(dx) $$ 
and 
$$ (F[\mu],\phi) = \int_{\R^+} (b(y,\cdot),\phi)    a(y)\,\mu(dy) - \rint a(y)\phi(y)\mu(dy) $$ 
for any test function $\phi$. 
Here $\kappa(x,y) $ is the rate at which individuals of size $x$ coalesce with individuals of size $y$, 
$a(y)$ is the global fragmentation rate of individuals of size $y$, 
and $b(y,\cdot)$ is a measure supported on $[0,y]$ such that $b(y,A)$ represents the probability a particle of size $y$ fragments to a particle with size in the Borel set $A$.

\begin{Def}
Given $T\geq 0$, we say a function $\mu \in C([0,T], \mathcal{M^+}(\R^+))$ is a \textit{weak solution} to \eqref{CoagModel} if for all $\phi \in (C^1 \cap W^{1,\infty})([0,T]\times \R^+)$ and for all $t\in [0,T]$, the following holds:
\begin{equation}\label{Def_Solution}
\begin{split}
& \int_{0}^{\infty} \phi(t,x)\mu_t(dx) - \int_{0}^{\infty}\phi(0,x)\mu_0(dx)=   \\
&\quad \int_0^t \int_{0}^{\infty} \left[ \p_t \phi(s,x) +g(s,\mu_s)(x)\p_x \phi(s,x) - d(s,\mu_s)(x)\phi(s,x)\right]\mu_s(dx)ds \\
 &\quad \quad +\int_0^t (K[\mu_s]+F[\mu_s],\phi(s,\cdot))\,ds  
+ \int_0^t \int_{0}^{\infty}  \phi(s,0)\beta(s,\mu_s)(x) \mu_s(dx) ds.  
\end{split}
\end{equation}
\end{Def}
For the numerical scheme, we will restrict ourselves to a finite domain, $[0,\xM]$. Thus, we impose the following assumptions on the growth, death and birth functions: 
\begin{enumerate}
\item[(A1)] For any $R>0$, there exists $L_R>0$ such that for all $\| \mu_i\|_{TV}\leq R$ and  $t_i\in [0,\infty)$ ($i=1,2$) the following hold for $f = g,d,\beta$
 \[ \|f(t_1, \mu_1) - f(t_2 , \mu_2)\|_{\infty} \leq L_R(|t_1 - t_2| + \| \mu_1 - \mu_2 \|_{BL}), \]

\item[(A2)]\label{zetaDef} There exists $\zeta>0$ such that for all $T>0$ 
\[\sup_{t\in [0,T]} \sup_{\mu \in \mathcal{M^+}(\R^+)} \| g(t,\mu) \|_{W^{1,\infty}} +\| d(t,\mu) \|_{W^{1,\infty}}+\| \beta(t,\mu) \|_{W^{1,\infty}}<\zeta,  \] 

 \item[(A3)] For all $(t,\mu)\in [0,\infty)\times \mathcal{M^+}(\R^+)$, 
 \[g(t,\mu)(0) >0 \quad \text{and} \quad g(t,\mu)(\xM)=0\]
 for some large $\xM>0$.
\end{enumerate}

We assume that the coagulation kernel $\kappa$ satisfies the following assumption:
 \begin{enumerate}
\item[(K1)] $\kappa$ is symmetric, nonnegative, bounded by a constant $C_\kappa$, and globally Lipschitz with 
		   Lipschitz constant $L_\kappa$. 
\item[(K2)]   \( \kappa(x,y) = 0  \text{ whenever } x+y > \xM.\)
\end{enumerate}

We assume that the fragmentation kernel satisfies the following assumptions:
\begin{itemize}
\item[(F1)] $a\in W^{1,\infty}(\R^+)$ is non-negative, 
\item[(F2)] for any $y\ge 0$,  $b(y,dx)$ is a measure such that 
\begin{itemize} 
\item[(i)] $b(y,dx)$ is  non-negative and  supported in $[0,y]$, and there exist a $C_b >0$ such  that $b(y,\R^+)\le C_b$ for all $y>0$, 

\item[(ii)] there exists $L_b$ such that for any $y,\bar y\ge 0$, 
\[ \|b(y,\cdot)-b(\bar y,\cdot)\|_{BL}\le L_b|y-\bar y|\]
\item[(iii)] for any $y\ge 0$, 
$$(b(y,dx),x) = \int_0^y x\,b(y,dx) = y. $$ 
\end{itemize}
\end{itemize}

The existence and uniqueness of mass conserving solutions of model \eqref{CoagModel} under these assumptions were established in \cite{AckLyonSaint2}.

\section{Numerical Method}\label{SecNumMeth}
We adopt the numerical discretization presented in \cite{AckLyonSaint2}.
For some fixed mesh sizes $\Delta x , \Delta t >0$, we discretize the size domain $[0,\xM]$ with the  cells  
\[\Lambda^{\Delta x}_j := [(j-\frac{1}{2})\Delta x, (j+\frac{1}{2})\Delta x),\text{ for } J= 1,\dots,J,\] \ and \[ \Lambda^{\Delta x}_0 := [0, \frac{\Delta x}{2}).\] We denote the midpoints of these grids by $x_j$.
The initial condition $\mu_0 \in \M^+(\R^+)$ will be approximated by a combination of Dirac measures
\[\mu_0^{\Delta x} = \sum_{j = 0}^J m_j^0 \delta_{x_j}, \text{ where } m_j^0 := \mu_0(\Lambda^{\Delta x}_j). \]

We first approximate the model coefficients  $\kappa$, $a$, $b$ as follow. For the physical ingredients, we define 
\[ a^{\Delta x}_i = \frac{1}{\Delta x}\int_{\Lambda^{\Delta x}_{i}} a(y)dy, \qquad 
\kappa^{\Delta x}_{i,j}=\frac{1}{\Delta x^2}\int_{\Lambda^{\Delta x}_{i}\times \Lambda^{\Delta x}_{j}} \kappa(x,y)dxdy \]
for $i,j\ge 1$, 
and
\[ a^{\Delta x}_0 = \frac{2}{\Delta x}\int_{\Lambda^{\Delta x}_0} a(y)dy, \qquad 
\kappa^{\Delta x}_{0,0}=\frac{4}{\Delta x^2}\int_{\Lambda^{\Delta x}_{0}\times \Lambda^{\Delta x}_{0}} \kappa(x,y)dxdy \]
(with the natural modifications for $\kappa^{\Delta x}_{0,j}$ and $\kappa^{\Delta x}_{i,0}$, $i\ge 1$). 
We then let $a^{\Delta x}\in W^{1,\infty}(\R^+)$ and $\kappa^{\Delta x}\in W^{1,\infty}(\R^+\times \R^+)$ be the linear interpolation of
the $a^{\Delta x}_i$ and $\kappa^{\Delta x}_{i,j}$ respectively. 
Finally, we define the measure $b^{\Delta x}(x_j,\cdot)\in \M^+(\Delta x\N)$ by 
\[b^{\Delta x}(x_j,\cdot) = \sum_{i\le j}  b(x_j,\Lambda^{\Delta x}_{i}) \delta_{x_j} =: \sum_{i\le j}b^{\Delta x}_{j,i}\delta_{x_j} \]
and then $b^{\Delta x}(x,\cdot)\in \M^+(\Delta x\N_0)$ for $x\ge 0$ as the linear interpolate between the $b^{\Delta x}(x_j,\cdot)$. When the context is clear, we omit the $\Delta x$ from the notation above.

We make use of these approximations to combine the high-resolution scheme presented in \cite{AckLyonSaint} with the fully explicit and semi-implicit schemes presented in \cite{AckLyonSaint3}. Together these schemes give us the numerical scheme

\begin{equation}\label{SO}
\left \lbrace
\begin{split}
m_j^{k+1} &= m_j^k -\frac{\Delta t}{\Delta x} (f_{j+\frac{1}{2}}^k - f_{j-\frac{1}{2}}^k) -\Delta t d_j^k m_j^k +\Delta t \left(\mathcal{C}_{j,k} + \mathcal{F}_{j,k} \right), \qquad j=1,..,J, \\
g_0^k m_0^{k} &= \Delta x \sum_{j=1}^{J}{}^{*} \beta_j^k m_j^k := \Delta x \left(\frac{3}{2}\beta_1^k m_1^k + \frac{1}{2}\beta_J^k m_J^k + \sum_{j=2}^{J-1} \beta_j^k m_j^k \right).
\end{split}\right. .
\end{equation}

Where the flux term is given by 
\begin{equation}\label{Def_HR_Flux}
f^k_{j+\frac{1}{2}} = \left \lbrace \begin{split}
&g_j^k m_j^k +\frac{1}{2}(g_{j+1}^k - g_{j}^k)m_j^k + \frac{1}{2} g_j^k  \; {\rm mm}(\Delta_+ m_j^k , \Delta_- m_j^k) & j = 2,3,\dots,J-2 \\
&g_j^k m_j^k & j = 0,1,J-1,J
\end{split}\right.  ,
\end{equation}
the fragmentation term, $\mathcal{F}_{j,k}$, is given by
\begin{equation}\label{Def_Fj}
\mathcal{F}_{j,k} := \sum_{i=j}^J b_{i,j}a_i m_i^k -a_jm_j^k ,
\end{equation} 
and the coagulation term, $\mathcal{C}_j$, is either given by an explicit discretization as 
\begin{equation}\label{Def_CjExp}
\mathcal{C}^{\text{exp}}_{j,k} := \frac{1}{2}\sum_{i=1}^{j-1} \kappa_{i,j-i}m_i^{k} m_{j-i}^k - \sumi \kappa_{i,j} m_i^k m_j^{k} , 
\end{equation}
or by an implicit one as
\begin{equation}\label{Def_CjImp}
\mathcal{C}^{\text{imp}}_{j,k} := \frac{1}{2}\sum_{i=1}^{j-1} \kappa_{i,j-i}m_i^{k+1} m_{j-i}^k - \sumi \kappa_{i,j} m_i^k m_j^{k+1} .
\end{equation}

As discussed in \cite{AckLyonSaint3}, the explicit and semi-implicit schemes behave differently with respect to the mass conservation and have different Courant–Friedrichs–Lewy (CFL) conditions. The assumed CFL condition for the schemes are 
\begin{equation}\label{CFL}
\begin{matrix}
\text{Explicit:} & \Delta t\Big(C_\kappa \|\mu_0 \|_{TV} \exp ((\zeta+  C_bC_a)T)  + C_a \max\{1,C_b\} + (1+\frac{3}{2\Delta x})\zeta\Big) \leq 1\\
\text{Semi-Implicit:} & \bar{\zeta} (2 + \frac{3}{2\Delta x})\Delta t \leq 1 ,
\end{matrix}
\end{equation}
where $\bar{\zeta}= \max\{ \zeta, \|a\|_{W^{1,\infty}}\}$, $C_a=\|a\|_\infty$. It is clear that the semi-implicit scheme has a less restrictive and simpler CFL condition than  the explicit scheme. In particular, the CFL condition of the semi-implicit scheme is independent on the initial condition unlike its counterpart. The trade off for this is a loss of qualitative behavior of the scheme in the sense of mass conservation. Indeed as shown in \cite{AckLyonSaint3}, when $\beta=d=g=0$, the semi-implicit coagulation term does not conserve mass whereas the explicit term does. 

It is useful to define the following coefficients:
\[A_j^k = \begin{cases}
g_j^k &j=1,J, \\
\frac{1}{2}\left( g_{j+1}^k +g_j^k +g_j^k \frac{  \; {\rm mm}(\Delta_+ m_j^k , \Delta_- m_j^k)}{\Delta_- m_j^k} \right) & j=2, \\
\frac{1}{2}\left(g_{j+1}^k +g_j^k +g_j^k \frac{  \; {\rm mm}(\Delta_+ m_j^k , \Delta_- m_j^k)}{\Delta_- m_j^k} -g_{j-1}^k \frac{  \; {\rm mm}(\Delta_- m_j^k , \Delta_- m_{j-1}^k)}{\Delta_- m_j^k} \right) &j=3,\dots,J-2, \\
\frac{1}{2}\left(2g_{j}^k  -g_{j-1}^k \frac{  \; {\rm mm}(\Delta_- m_j^k , \Delta_- m_{j-1}^k)}{\Delta_- m_j^k}\right)  & j=J-1,
\end{cases} \] 
and
\[B_j^k = \begin{cases}
\Delta_- g_j^k & j=1,J, \\
\frac{1}{2} \Delta_+ g_j^k & j=2, \\
\frac{1}{2}(\Delta_+ g_j^k + \Delta_-g_j^k) & j=3,\dots,J-2, \\
\frac{1}{2} \Delta_- g_j^k  & j=J-1.
\end{cases}.\]
Notice, $|A_j^k| \leq \frac{3\Delta t}{2\Delta x}\zeta$ and $A_j^k -B_j^k \geq 0$ as
\[ 2(A_j^k - B_j^k) = \begin{cases} 
2g_{j-1}^k & j=1,J, \\
g_j^k \left( 2+\frac{ \; {\rm mm}(\Delta_+ m_j^k , \Delta_- m_j^k)}{\Delta_- m_j^k}\right) & j=2,\\
g_j^k \left( 1+\frac{ \; {\rm mm}(\Delta_+ m_j^k , \Delta_- m_j^k)}{\Delta_- m_j^k}\right) + g_{j-1}^k \left( 1-\frac{ \; {\rm mm}(\Delta_- m_j^k , \Delta_- m_{j-1}^k)}{\Delta_- m_j^k} \right) & j=3,\dots ,J-2, \\
g_j^n + g_{j-1}^n \left(1-\frac{ \; {\rm mm}(\Delta_- m_j^n,\Delta_- m_{j-1}^n)}{\Delta_- m_j^n}  \right) &j=J-1.
\end{cases}.   \]
Scheme (\ref{SO}) can then be rewritten as 
\begin{equation}\label{SO2}
\left \lbrace
\begin{split}
m_j^{k+1} &= (1-\frac{\Delta t}{\Delta x}A_j^k - \Delta t (d_j^k + a_j))m_j^k + \frac{\Delta t}{\Delta x}(A_j^k - B_j^k)m_{j-1}^k \\ 
& \qquad + \Delta t \sum_{i=j}^J b_{i,j}a_i m_i^k +\Delta t. \mathcal{C}_{j,k}\\
g_0^k m_0^{k} &= \Delta x \sum_{j=1}^{J}{}^* \beta_j^k m_j^k  \; .
\end{split}. \right.
\end{equation}
Depending on the choice of coagulation term, this formulation leads to either

\begin{equation}\label{SO2Exp}
\left \lbrace
\begin{split}
m_j^{k+1} &= (1-\frac{\Delta t}{\Delta x}A_j^k - \Delta t (d_j^k + a_j)- \Delta t \sum_{i=1}^J \kappa_{i,j}m_i^k )m_j^k + \frac{\Delta t}{\Delta x}(A_j^k - B_j^k)m_{j-1}^k \\ 
& \qquad + \Delta t \sum_{i=j}^J b_{i,j}a_i m_i^k +\frac{\Delta t}{2} \sum_{i=1}^{j-1} \kappa_{i,j-i} m_i^k m_{j-i}^k\\
g_0^k m_0^{k} &= \Delta x \sum_{j=1}^{J}{}^* \beta_j^k m_j^k  \; 
\end{split}, \right.
\end{equation} 
for the explicit term, $\mathcal{C}^{\text{exp}}_{j,k}$ or 
\begin{equation}\label{SO2Imp}
\left \lbrace
\begin{split}
(1+ \Delta t \sumi \kappa_{i,j}m_i^k)m_j^{k+1} &= (1-\frac{\Delta t}{\Delta x}A_j^k - \Delta t (d_j^k + a_j))m_j^k + \frac{\Delta t}{\Delta x}(A_j^k - B_j^k)m_{j-1}^k \\ 
& \qquad + \Delta t \sum_{i=j}^J b_{i,j}a_i m_i^k +\frac{\Delta t}{2} \sum_{i=1}^{j-1} \kappa_{i,j-i} m_i^{k+1} m_{j-i}^k\\
g_0^k m_0^{k} &= \Delta x \sum_{j=1}^{J}{}^* \beta_j^k m_j^k  \; ,
\end{split}. \right.
\end{equation}
for the implicit term, $\mathcal{C}^{\text{imp}}_{j,k}$.

For these, schemes, we have the following Lemmas which are proven in the appendix:
\begin{lem}\label{ThrmTVBdPos}
For each $k = 1,2,\dots,\bar{k}$, 
\begin{itemize}
\item $m_j^k \geq 0$ for all $j = 1,2,\dots J$,
\item $\| \mu^k_{\Delta x}\|_{TV} \leq \|\mu_0 \|_{TV} \exp((\zeta + C_b C_a)T).$
\end{itemize}
\end{lem}

\begin{lem}\label{ThrmLip}
For any $l,p = 1,2,\dots, \bar{k}$,
\[\|\mu_{\Delta x}^l - \mu_{\Delta x}^p \|_{BL} \leq \mathcal{L}_T |l-p|. \]
\end{lem}

Using the above two Lemmas, we can arrive at analogous  results for the linear interpolation \eqref{LinInter}:
\begin{equation}\label{LinInter}
\mu_{\Delta x}^{\Delta t}(t):= \mu_{\Delta x}^0 \chi_{\{0\}}(t) + \sum_{k=0}^{\bar{k}-1} \left[ (1- \frac{t-k\Delta t}{\Delta t})\mu^k_{\Delta x} + \frac{t-k\Delta t}{\Delta t} \mu^{k+1}_{\Delta x} \right] \chi_{(k\Delta t, (k+1)\Delta t]}(t). 
\end{equation}
Thus by the well know Ascoli-Arzela Theorem, we have the existence of a convergent subsequence of the net $\{\mu_{\Delta x}^{\Delta t}(t) \}$ in $C([0,T],\M^+([0,\xM])$. We now need only show any convergent subsequence converges to the unique solution \eqref{Def_Solution}. 

\begin{thm}\label{ThmConvergence}
As $\Delta x,\Delta t\to 0$ the sequence $\mu_{\Delta x}^{\Delta t}$ converges  in $C([0,T],\mathcal{M^+}([0,\xM]))$ to the solution of \eqref{CoagModel}.
\end{thm}

\begin{proof}
By multiplying \eqref{SO} by a superfluously smooth test function $\phi \in (W^{1,\infty} \cap C^2)([0,T]\times \R)$, denoting $\phi_j^k := \phi (k\Delta t,x_j)$, summing over all $j$ and $k$, and rearranging we arrive at
\begin{align}
\sum_{k=0}^{\bar{k}-1} \sum_{j =1}^{J} \left((m_j^{k+1}-m_j^k)\phi_j^k + \frac{\Delta t}{\Delta x} (f_{j+\frac{1}{2}}^k - f_{j-\frac{1}{2}}^k)\phi_j^k \right) 
+\Delta t \sum_{k=0}^{\bar{k}-1} \sum_{j =1}^{\infty} d_j^k m_j^k \phi_j^k  \label{SumUp2}  \\ 
= \Delta t \sum_{k=1}^{\bar{k}-1}\sum_{j=1}^J  \phi_j^k \left( \frac{1}{2}\sum_{i=1}^{j-1} \kappa_{i,j-i}m_i^{k} m_{j-i}^k - \sumi \kappa_{i,j} m_i^k m_j^{k} + \sum_{i=j}^J b_{i,j}a_i m_i^k -a_jm_j^k\right). \notag
\end{align}   

The left-hand side of equation \eqref{SumUp2} was shown in \cite{AckLyonSaint} to be equivalent to
\begin{equation*}
\begin{split}
& \int_{0}^{\xM} \phi(T,x)d\mu^{\bar{k}}_{\Delta x}(x) 
- \int_{0}^{\xM} \phi(0,x)d\mu^0_{\Delta x}(x) \\ 
& - \Delta t\sum_{k=0}^{\bar{k}-1} \left( \int_{0}^{\xM} \partial_t\phi(t_k,x)d\mu^k_{\Delta x}(x) 
+ \int_{0}^{\xM} \partial_x\phi(t_k,x)g(t_k,\mu^k_{\Delta x})(x) d\mu^k_{\Delta x}(x) \right. \\
& \left. \hspace{0.5cm} -\int_{\R^+} d(t_k,\mu_{\Delta x}^k)(x) \phi(t_k,x)d\mu_{\Delta x}^k(x) 
	+\int_{0}^{\xM} \phi(t_k,\Delta x)\beta(t_k,\mu^k_{\Delta x})(x) d\mu_{\Delta x}^k(x)\right) + o(1),  
\end{split}
\end{equation*}
where $o(1) \longrightarrow 0$ as $\Delta t, \Delta x \longrightarrow 0$.

The right-hand side of \eqref{SumUp2} was shown in \cite{AckLyonSaint3} to be equal to 
\[\Delta t \sum_{k=1}^{\bar{k}-1} \Big\{ (K[\mu_{\Delta x}^{\Delta t}(t_k)],\phi(t_k,\cdot))+(F[\mu_{\Delta x}^{\Delta t}(t_k)], \phi(t_k,\cdot))\Big\} + O(\Delta x).\]

Making use of results, it is then easy to see \eqref{SumUp2} is equivalent to
\begin{align*}
& \int_{0}^{\xM} \phi(T,x )d\mu^{\Delta t}_{\Delta x}(T)(x) 
-\int_{0}^{\xM} \phi(0,x)d\mu^0_{\Delta x}(x) \\ 
& = 
\int_0^T \left( \int_{0}^{\xM} \partial_t\phi(t,x) 
+ \partial_x\phi(t,x) g(t,\mu_{\Delta x}^{\Delta t} (t))(x) d\mu^{\Delta t}_{\Delta x}(t)(x) \right. \\  
& \left. \hspace{0.5cm} -\int_{0}^{\xM} d(t,\mu_{\Delta x}^{\Delta t} (t))(x) \phi(t,x)d\mu^{\Delta t}_{\Delta x}(t)(x) 
+\int_{0}^{\xM} \phi(t,\Delta x)\beta(t,\mu_{\Delta x}^{\Delta t} (t))(x) d\mu^{\Delta t}_{\Delta x}(t)(x) \right) dt\\
&\quad +\int_0^T (K[\mu_{\Delta x}^{\Delta t}(t)],\phi(t,\cdot)) + (F[\mu_{\Delta x}^{\Delta t}(t)], \phi(t,\cdot)) \,dt + o(1). 
\end{align*}

Passing the limit as $\Delta t, \Delta x \longrightarrow 0$ along a converging subsequence, 
we then obtain that equation \eqref{Def_Solution} 
holds for any  
$\phi \in (C^2\cap W^{1,\infty})([0,T]\times\R^+)$ with compact support. 
A standard density argument  shows that  
equation \eqref{Def_Solution} holds for any  
$\phi \in (C^1\cap W^{1,\infty})([0,T]\times\R^+)$. As the weak solution is unique \cite{AckLyonSaint2}, we conclude the net $\{\mu_{\Delta x}^{\Delta t}\}$ converges to the solution of model \eqref{CoagModel}.
\end{proof}

We point out that while these schemes are higher-order in space, they are only first order in time. To lift these schemes into a second-order in time as well, we make use of the second-order Runge-Kutta time discretization \cite{Shu} for the explicit scheme and second-order Richardson extrapolation \cite{Richardson} for the semi-implicit scheme. 


\section{Numerical Examples}\label{Sec_Numerics}

In this section, we provide numerical simulations which test the order of the explicit and semi-implicit schemes developed in the previous sections.
For each example, we give the BL error and the order of convergence. 
To appreciate the gain in the order of convergence compared to those studied in \cite{AckLyonSaint3} which are based on first order approximation of the transport term, 
we add some of the numerical results from the scheme presented in \cite{AckLyonSaint3}. 

In some of the following examples, the exact solution of the model problem is given. In these cases, we approximate the order of accuracy, $q$, with the standard calculation: \[q=\log_2\left(\dfrac{\rho (\mu^{\Delta t}_{\Delta x}(T),\mu(T))}{\rho (\mu^{0.5\Delta t}_{0.5\Delta x}(T),\mu(T))} \right)  \] where $\mu$ represents the exact solution of the examples considered. In the cases where the exact solutions are unknown, we approximate the order by 
\[q=\log_2\left(\dfrac{\rho (\mu^{\Delta t}_{\Delta x}(T),\mu^{2\Delta t}_{2\Delta x}(T))}{\rho (\mu^{0.5\Delta t}_{0.5\Delta x}(T),\mu^{\Delta t}_{\Delta x}(T))} \right)  \]
and we report the numerator of the log argument as the error.
The metric $\rho$ we use  here was introduced in \cite{Jab} and is equivalent to the BL metric, namely 
\[ C\rho(\mu,\nu) \leq \|\mu -\nu \|_{BL} \leq \rho(\mu,\nu)\] 
for some constant $C$ (dependent on the finite domain). 
As discussed in \cite{Jab}, this metric is more efficient to compute than the BL norm and maintains the same order of convergence. 
An alternative to this algorithm would be to make use of the algorithms presented in \cite{HilleTheewis} where convergence in the Fortet-Mourier distance is considered.

\paragraph{Example 1} In this example, we test the quality of the finite difference schemes against coagulation equations. To this end, we take $\kappa (x,y) \equiv 1$ and $\mu_0 = e^{-x} dx$ with all other ingredients set to $0$. This example has exact solution given by \[ \mu_t = \left(\frac{2}{2+t}\right)^2 \exp\left(-\frac{2}{2+t}x \right)dx\] see \cite{KeckBortz} for more details. The simulation is performed over the finite domain $(x,t) \in [0,20]\times [0,0.5]$. We present the BL error and the numerical order of convergence for both schemes in Table \ref{TableCoagOnly}.

\begin{table}[h]
\center \begin{tabular}{|cc|ccc|ccc|}
\hline
\multicolumn{2}{|c}{}                          & \multicolumn{2}{c|}{\textbf{Explicit}}                                                              & \multicolumn{2}{c|}{\textbf{Semi-Implicit}}                                                         \\ \hline
\multicolumn{1}{|c|}{\textbf{Nx}} & \textbf{Nt} & \multicolumn{1}{c|}{\textbf{BL Error}} & \multicolumn{1}{c|}{\textbf{Order}} & \multicolumn{1}{c|}{\textbf{BL Error}} & \multicolumn{1}{c|}{\textbf{Order}}  \\ \hline
\multicolumn{1}{|c|}{100}         & 250         & \multicolumn{1}{c|}{0.0020733}         & \multicolumn{1}{c|}{}                    & \multicolumn{1}{c|}{0.0020886}         & \multicolumn{1}{c|}{}           \\ \hline
\multicolumn{1}{|c|}{200}         & 500         & \multicolumn{1}{c|}{0.00054068}        & \multicolumn{1}{c|}{1.9391}                     & \multicolumn{1}{c|}{0.00054408}        & \multicolumn{1}{c|}{1.9407}                \\ \hline
\multicolumn{1}{|c|}{400}         & 1000        & \multicolumn{1}{c|}{0.00013802}        & \multicolumn{1}{c|}{1.9699}                 & \multicolumn{1}{c|}{0.00013883}        & \multicolumn{1}{c|}{1.9705}                      \\ \hline
\multicolumn{1}{|c|}{800}         & 2000        & \multicolumn{1}{c|}{3.4842e-05}        & \multicolumn{1}{c|}{1.9860}                     & \multicolumn{1}{c|}{3.5040e-05}         & \multicolumn{1}{c|}{1.9862}                      \\ \hline
\multicolumn{1}{|c|}{1600}        & 4000        & \multicolumn{1}{c|}{8.7417e-06}        & \multicolumn{1}{c|}{1.9948}                       & \multicolumn{1}{c|}{8.7906e-06}        & \multicolumn{1}{c|}{1.9950}                 \\ \hline
\multicolumn{2}{|c|}{}                          & \multicolumn{2}{c|}{\textbf{Explicit (1st order)}}                                                              & \multicolumn{2}{c|}{\textbf{Semi-Implicit (1st order)}}                                                         \\ \hline
\multicolumn{1}{|c|}{800}        & 2000        & \multicolumn{1}{c|}{0.015675}        & \multicolumn{1}{c|}{0.96974}         & \multicolumn{1}{c|}{0.010996}        & \multicolumn{1}{c|}{0.97418}                        \\ \hline
\end{tabular}
\caption{\label{TableCoagOnly} Error and order of convergence for example 1. Here Nx and Nt represent the number of points in $x$ and $t$, respectively.
The numerical result in the last row for the 1st order variant is generated from the scheme presented in \cite{AckLyonSaint3}.}
\end{table}

\paragraph{Example 2} In this example, we test the quality of the finite difference scheme against fragmentation only equations. We point out that in this case, the two schemes are identical in the spacial component. For this demonstration, we take $\mu_0 = e^{-x}dx$, $ b(y,\cdot) = \frac{2}{y}dx \text{ and } a(x) = x$. As given in \cite{SingSahaKum}, this problem has an exact solution of 
\[\mu_t = (1+t)^2\exp(-x(1+t))dx .\]
The simulation is performed over the finite domain $(x,t) \in [0,20]\times [0,0.5]$.
We present the BL error and the numerical order of convergence for both schemes in Table \ref{TableFragOnly}. Note as compared to coagulation, the fragmentation process is more affected by the truncation of the domain. This results in the numerical order of the scheme being further from 2 than example 1. 

\begin{table}[h]
\center 
\begin{tabular}{|cc|ccc|ccc|}
\hline
\multicolumn{2}{|c}{}                          & \multicolumn{2}{c|}{\textbf{Explicit}}                                                              & \multicolumn{2}{c|}{\textbf{Semi-Implicit}}                                                         \\ \hline
\multicolumn{1}{|c|}{\textbf{Nx}} & \textbf{Nt} & \multicolumn{1}{c|}{\textbf{BL Error}} & \multicolumn{1}{c|}{\textbf{Order}}  & \multicolumn{1}{c|}{\textbf{BL Error}} & \multicolumn{1}{c|}{\textbf{Order}} \\ \hline
\multicolumn{1}{|c|}{100}         & 250         & \multicolumn{1}{c|}{0.0053857}         & \multicolumn{1}{c|}{}                    & \multicolumn{1}{c|}{0.0053836}         & \multicolumn{1}{c|}{}                   \\ \hline
\multicolumn{1}{|c|}{200}         & 500         & \multicolumn{1}{c|}{0.0014548}         & \multicolumn{1}{c|}{1.8883}                    & \multicolumn{1}{c|}{0.0014536}         & \multicolumn{1}{c|}{1.8890}                     \\ \hline
\multicolumn{1}{|c|}{400}         & 1000        & \multicolumn{1}{c|}{0.00037786}        & \multicolumn{1}{c|}{1.9449}                      & \multicolumn{1}{c|}{0.00037753}        & \multicolumn{1}{c|}{1.9449}                     \\ \hline
\multicolumn{1}{|c|}{800}         & 2000        & \multicolumn{1}{c|}{9.6317e-05}        & \multicolumn{1}{c|}{1.9720}                    & \multicolumn{1}{c|}{9.6322e-05}        & \multicolumn{1}{c|}{1.9707}                  \\ \hline
\multicolumn{1}{|c|}{1600}        & 4000        & \multicolumn{1}{c|}{2.4468e-05}        & \multicolumn{1}{c|}{1.9769}               & \multicolumn{1}{c|}{2.4514e-05}        & \multicolumn{1}{c|}{1.9743}                  \\ \hline
\multicolumn{2}{|c}{}                          & \multicolumn{2}{c|}{\textbf{Explicit (1st order)}}                                                              & \multicolumn{2}{c|}{\textbf{Semi-Implicit (1st order)}}                                                         \\ \hline
\multicolumn{1}{|c|}{800}        & 2000        & \multicolumn{1}{c|}{0.059804}        & \multicolumn{1}{c|}{0.9128}                      & \multicolumn{1}{c|}{0.096943}        & \multicolumn{1}{c|}{0.86667}                     \\ \hline
\end{tabular}
\caption{\label{TableFragOnly}Error and order of convergence for example 2. Here Nx and Nt represent the number of points in $x$ and $t$, respectively.
The numerical result in the last row for the 1st order variant is generated from the scheme presented in \cite{AckLyonSaint3}.}
\end{table}

\paragraph{Example 3} In this example, we test the schemes against the complete model i.e. with all biological and physical processes. To this end, we take $\mu_0 = e^{-x}dx$, $g(x) = 2-2e^{x-20}$, $\beta (x) = 2$, $d(x) = 1$, $\kappa (x,y) = 1$, $a(x) = x$, and $b(y,\cdot) = \frac{2}{y}$. The simulation is performed over the finite domain $(x,t) \in [0,20]\times [0,0.5]$. To our knowledge, the solution of this problem is unknown.

\begin{table}[h]
\center 
\begin{tabular}{|cc|ccc|ccc|}
\hline
\multicolumn{2}{|c}{}                          & \multicolumn{2}{c|}{\textbf{Explicit}}                                                              & \multicolumn{2}{c|}{\textbf{Semi-Implicit}}                                                         \\ \hline
\multicolumn{1}{|c|}{\textbf{Nx}} & \textbf{Nt} & \multicolumn{1}{c|}{\textbf{BL Error}} & \multicolumn{1}{c|}{\textbf{Order}} & \multicolumn{1}{c|}{\textbf{BL Error}} & \multicolumn{1}{c|}{\textbf{Order}} \\ \hline
\multicolumn{1}{|c|}{100}         & 250         & \multicolumn{1}{c|}{0.0023026}         & \multicolumn{1}{c|}{}                        & \multicolumn{1}{c|}{0.0028799}                 & \multicolumn{1}{c|}{     }               \\ \hline
\multicolumn{1}{|c|}{200}         & 500         & \multicolumn{1}{c|}{0.00085562}        & \multicolumn{1}{c|}{1.4282}                       & \multicolumn{1}{c|}{0.00076654       }                  & \multicolumn{1}{c|}{1.9096       }               \\ \hline
\multicolumn{1}{|c|}{400}         & 1000        & \multicolumn{1}{c|}{0.0002743}         & \multicolumn{1}{c|}{1.6412}                    & \multicolumn{1}{c|}{0.00076654}                  & \multicolumn{1}{c|}{1.9549}                            \\ \hline
\multicolumn{1}{|c|}{800}         & 2000        & \multicolumn{1}{c|}{7.5404e-05}        & \multicolumn{1}{c|}{1.8631}                    & \multicolumn{1}{c|}{5.021e-05}                  & \multicolumn{1}{c|}{1.9775}                                \\ \hline
\multicolumn{1}{|c|}{1600}           & 4000           & \multicolumn{1}{c|}{1.9495e-05}                 & \multicolumn{1}{c|}{1.9515   }                      & \multicolumn{1}{c|}{1.2651e-05}                  & \multicolumn{1}{c|}{1.9887}                     \\ \hline
\multicolumn{2}{|c}{}                          & \multicolumn{2}{c|}{\textbf{Explicit (1st order)}}                                                              & \multicolumn{2}{c|}{\textbf{Semi-Implicit (1st order)}}                                                         \\ \hline
\multicolumn{1}{|c|}{800}        & 2000        & \multicolumn{1}{c|}{0.0092432}        & \multicolumn{1}{c|}{0.97728}                       & \multicolumn{1}{c|}{0.0014192}        & \multicolumn{1}{c|}{0.98355}           \\ \hline
\end{tabular}
\caption{\label{TableEverything}Error and order of convergence for example 3. Here Nx and Nt represent the number of points in $x$ and $t$, respectively.
The numerical result in the last row for the 1st order variant is generated from the scheme presented in \cite{AckLyonSaint3}.}
\end{table}

\paragraph{Example 4} As mentioned in \cite{AckLyonSaint3}, the mixed discrete and continuous fragmentation model studied in \cite{BairdSuli1, BairdSuli2}, with adjusted assumptions, is a special case of model \eqref{CoagModel}. Indeed, by removing the biological and coagulation terms and letting the kernel
\[(b(y,\cdot),\phi) = \sum_{i=1}^N b_i(y) \phi(ih) + \int_{Nh}^y \phi(x) b^c(y,x) dx \]
with $\support \ b^c(y,\cdot) \subset [Nh,y]$ for some $h>0$, we have the mixed model in question. We wish to demonstrate the finite difference scheme presented here maintains this mixed structure.  

To this end, we take the fragmentation kernel 
\[ b^c(y,x) = \frac{2}{y}, \quad b_i(y) =  \frac{2}{y}, \text{ and } a(x) = x^{-1},\] with initial condition $\mu = \sum_{i=1}^5 \delta_i + \chi_{[5,15]}(x)dx$, where $\chi_A$ represents the characteristic function over the set $A$. This is similar to some examples in \cite{BairdSuli2} where more detail and analysis are provided. In Figure \ref{fig:MixedFrag}, we present the simulation of this example. Notice, the mixed structure is preserved in finite time. For examples of this type, the scheme could be improved upon by the inclusion of mass conservative fragmentation terms similar to those presented in \cite{AckLyonSaint2}.

\begin{figure}[h!]
    \centering
    \includegraphics[scale=0.35]{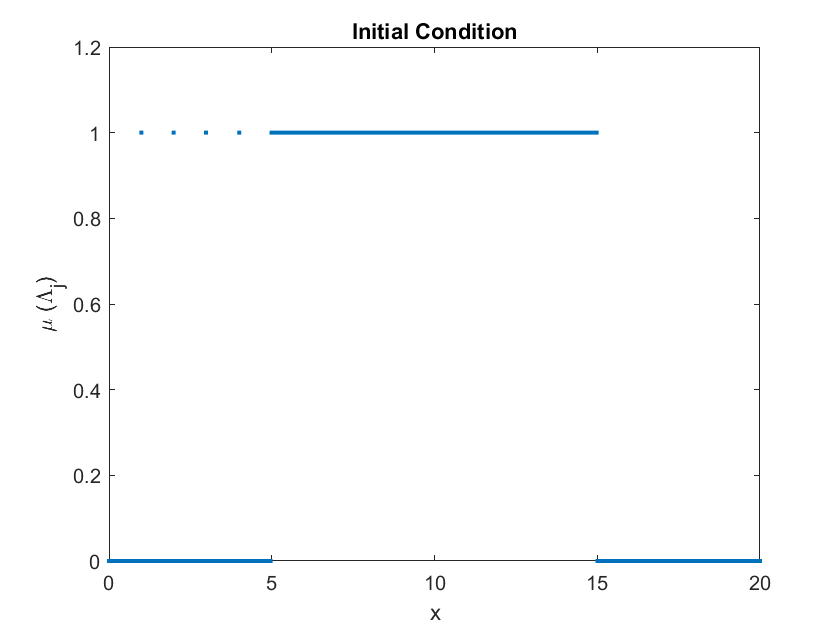}
    \includegraphics[scale=0.35]{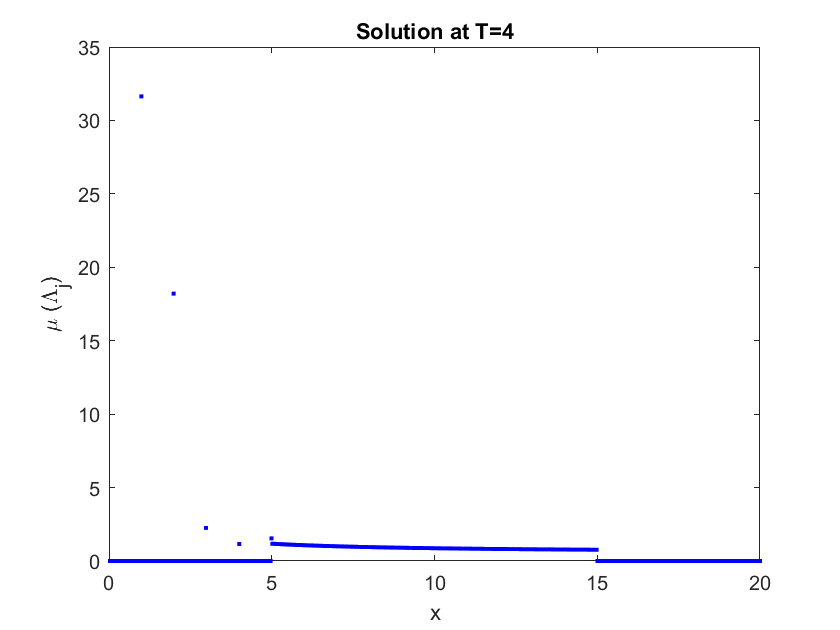}
    \caption{Initial condition and numerical solution at time $T=4$ of example 4.}
    \label{fig:MixedFrag}
\end{figure}

\section{Conclusion}\label{Sec_Conclusion}

In this paper, we have lifted two of the first order finite difference schemes presented in \cite{AckLyonSaint3} to second order high resolution schemes using flux limiter methods. The difference between both schemes is only found in the coagulation term where the semi-implicit scheme is made linear. In context of standard structured population models (i.e. without coagulation or fragmentation), these type of schemes have been shown to be well-behaved in the presences of discontinuities and singularities. This quality makes them a well fit tool for studying PDEs in spaces of measures. We prove the convergence of both schemes under the assumption of natural CFL conditions. The order of convergence of both schemes is then tested numerically with previously used examples. 

In summary, the schemes preform as expected in the presence of smooth initial conditions. In all such simulations, the numerical schemes presented demonstrate a convergence rate of order 2. For simulations with biological terms, this convergence rate is expected to drop when singularities and discontinuities occur as demonstrated in \cite{AckLyonSaint}. Mass conservation of the schemes, an important property for coagulation/fragmentation processes, is discussed in detail in \cite{AckLyonSaint2,AckLyonSaint3}.

\paragraph{Acknowledgments:} The research of ASA is supported in part by funds from R.P. Authement Eminent Scholar and Endowed Chair in Computational Mathematics at the University of Louisiana at Lafayette. RL is grateful for the support of the Carl Tryggers Stiftelse via the grant CTS 21--1656.

\section{Appendix}\label{Sec_Append}
\subsection{Proof of Lemmas \ref{ThrmTVBdPos} and \ref{ThrmLip} }
In this section, we present the proofs of Lemmas \ref{ThrmTVBdPos} and \ref{ThrmLip} for the explicit coagulation term. The semi-implicit term follows from similar arguments in the same fashion as \cite{AckLyonSaint3}. 

\paragraph{Proof of Lemma \ref{ThrmTVBdPos}}
\begin{proof}
We first prove via induction that for any $k = 1,2,\dots, \bar{k},$ $\mu_{\Delta x}^k$ satisfies the following:
\begin{itemize}
\item[(i)] $\mu^k_{\Delta x} \in \M^+(\R^+)$ i.e. $m_j^k \geq 0$ for all $j =1,\dots, J$, 
\item[(ii)] $\|\mu_{\Delta x}^k\|_{TV} \leq \|\mu^0_{\Delta x}\|_{TV}  (1+(\zeta + C_b C_a )\Delta t)^k .$
\end{itemize}
Then, the TV bound in the Lemma follows from standard arguments (see e.g. Lemma 4.1 in \cite{AckLyonSaint3}). We prove this Theorem for the choice of the explicit coagulation term, $\mathcal{C}_{j,k}^{\text{exp}},$ as the implicit case is similar and more straight forward. 

We begin by showing that $m_j^{k+1} \geq 0$ for every $j = 1,2, \dots, J$. Notice by way of \eqref{SO2Exp}, this reduces down to showing 
\begin{equation*}
\frac{\Delta t}{\Delta x}A_j^k + \Delta t (d_j^k + a_j)+ \Delta t \sum_{i=1}^J \kappa_{i,j}m_i^k  \leq 1.
\end{equation*}
Indeed, by the CFL condition \eqref{CFL}, induction hypothesis, and  
\[ \sumi\kappa_{i,j}m_i^k \le C_\kappa \sumi m_i^k = C_\kappa \|\mu^k_{\Delta x}\|_{TV}
\le C_\kappa \|\mu^0_{\Delta x}\|_{TV} \exp((\zeta + C_b C_a )T), \] 
we arrive at the result.

For the TV bound, we have since the $m_j^k$ are non-negative, $\|\mu_{\Delta x}^k\|_{TV} = \sumj m_j^k$. By rearranging \eqref{SO2Exp} and summing over $j = 1,2,\dots, J$ we have
\begin{equation}\label{Equ100}
\begin{split}
\|\mu^{k+1}_{\Delta x}\|_{TV} & 
\le \sumj m^{k}_j + \frac{\Delta t}{\Delta x}\sumj \Big(f^k_{j-\frac12}-f^k_{j+\frac12} \Big)
+ \Delta t\sumj\sum_{i=j}^J b_{i,j}a_i m_i^k   \\
& \hspace{2cm}\left. \qquad+ \Delta t \Big(\frac{1}{2} \sumj \sum_{i=1}^{j-1} \kappa_{i,j-i}m_i^{k} m_{j-i}^k 
  - \sumj\sumi \kappa_{i,j} m_i^k m_j^{k} \Big).
\right.
\end{split}
\end{equation}
To bound the right-hand side of equation \eqref{Equ100}, we directly follow the arguments of Lemma 4.1 in \cite{AckLyonSaint3} which yields 
\[  \|\mu^{k+1}_{\Delta x}\|_{TV}\le (1+(\zeta+C_aC_b)\Delta t)\sumj m^k_j 
= (1+(\zeta+C_aC_b)\Delta t) \|\mu^k_{\Delta x}\|_{TV}. \]
Using the induction hypothesis, we obtain 
$  \|\mu^{k+1}_{\Delta x}\|_{TV}\le  \|\mu^0_{\Delta x}\|_{TV} (1+(\zeta + C_b C_a )\Delta t)^{k+1}$ 
as desired.
\end{proof}

\paragraph{Proof of Lemma \ref{ThrmLip}}

\begin{proof}
For $\phi \in W^{1,\infty}(\R^+)$ with $\|\phi\|_{W^{1,\infty}} \leq 1$, 
and denoting $\phi_j:=\phi(x_j)$, we have for any $k$, 
\begin{align*}
(\mu^{k+1}_{\Delta x}-\mu^k_{\Delta x}, \phi) 
=&  \sumj  (m^{k+1}_{j}-m^k_{j}) \phi_j \\
\leq& \Delta t \sumj  \phi_j \Big(\frac{1}{\Delta x} (f_{j-\frac12}^k   -f_{j+\frac12}^k)
-  d_j^km_j^k - a_jm_j^k \\
& \quad +\frac{1}{2} \sum_{i=1}^{j-1}\kappa_{i,j-i}m_i^{k} m_{j-i}^k - \sumi \kappa_{i,j} m_i^k m_j^{k}  
+ \sum_{i=j}^J b_{i,j}a_i m_i^k\Big).
\end{align*}
Let $C$ be the right-hand side of the TV-bound from Lemma \ref{ThrmTVBdPos}, we then see 
\[ (\mu^{k+1}_{\Delta x}-\mu^k_{\Delta x}, \phi) 
\le \frac{\Delta t}{\Delta x} \sumj  \phi_j (f_{j-\frac12}^k   -f_{j+\frac12}^k)
+ \Delta t (\zeta+C_a+C_bC_a+\frac32 C_\kappa C^* )C^*.  
\] 
Moreover, since $g_J^k=0$ the sum in the right-hand side takes the form 
\begin{eqnarray*}
\phi_1g_0^km_0^k + \sum_{j=1}^{J-1} (\phi_{j+1}-\phi_j) f_{j+\frac12}^k  
= \Delta x \phi_1 \sumj {}^* \beta_j^k m_j^k + \sum_{j=1}^{J-1} (\phi_{j+1}-\phi_j) f_{j+\frac12}^k
\le 3.5\Delta x\zeta C^*. 
\end{eqnarray*}
We thus obtain 
\begin{align*}
(\mu^{k+1}_{\Delta x}-\mu^k_{\Delta x}, \phi) \le L\Delta t,\qquad L:=(3.5\zeta+C_a+C_bC_a+\frac32 C_\kappa C^* )C^*. 
\end{align*}
Taking the supremum over $\phi$ gives $\|\mu^{k+1}_{\Delta x}-\mu^k_{\Delta x}\|_{BL}\le L\Delta t$ for any $k$. 
The result follows. 
\end{proof}

\bibliographystyle{apa}

\end{document}